\documentclass[a4paper,12pt]{amsart}
\usepackage{amsfonts,amsmath,amssymb,euscript,inputenc}
\usepackage{graphicx,psfrag,subfigure,xcolor}

\title[Wave equation in rank one]
{The shifted wave equation\\
on Damek--Ricci spaces\\
and on homogeneous trees}

\author[J.--Ph. Anker, P. Martinot, E. Pedon, A.G. Setti]
{Jean--Philippe Anker, Pierre Martinot,\\
Emmanuel Pedon \& Alberto G. Setti}

\address{Jean--Philippe Anker,
Universit\'e d'Orl\'eans \& CNRS,
F\'ed\'eration Denis Poisson (FR 2964) \& Laboratoire MAPMO (UMR 6628),
B\^atiment de math\'ematiques -- Route de Chartres,
B.P. 6759 -- 45067 Orl\'eans cedex 2 -- France}

\email{anker@univ-orleans.fr}

\address{Pierre Martinot,
2 rue de la Maladrie,
54630 Flavigny--sur--Moselle, France}

\email{pierremartinot@orange.fr}

\address{Emmanuel Pedon,
Universit\'e de Reims (Champagne-Ardenne),
F\'ed\'eration ARC (FR 3399)
\& Laboratoire de Math\'ematiques (EA 4535),
Moulin de la Housse,
B.P.~1039,
51687 Reims Cedex 2,
France}

\email{emmanuel.pedon@univ-reims.fr}

\address{Alberto G. Setti,
Universit\`a degli Studi dell'Insubria,
Dipartimento di Scienza e Alta Tecnologia,
Sezione di Matematica,
via Valleggio 11,
22100 Como,
Italia}

\email{alberto.setti@uninsubria.it}

\date{\today}



\thanks{Research partially supported by the European Commission
(HCM Network \textit{Fourier Analysis\/} 1994--1997,
TMR Network \textit{Harmonic Analysis\/} 1998--2002,
IHP Network \textit{HARP\/} 2002--2006)
and by a French-Italian research project
(PHC Galil\'ee 25970QB \textit{VAMP\/} 2011--2012).}

\newtheorem{lemma}{Lemma}[section]
\newtheorem{theorem}[lemma]{Theorem}
\newtheorem{corollary}[lemma]{Corollary}
\newtheorem{remark}[lemma]{Remark}

\newcommand{\A}{\mathcal{A}}
\newcommand{\C}{\mathbb{C}}
\newcommand\const{\operatorname{const.}}
\newcommand{\E}{\mathcal{E}}
\newcommand{\F}{\mathcal{F}\ssb}
\renewcommand{\H}{\mathcal{H}}

\newcommand{\K}{\mathcal{K}}
\renewcommand{\L}{\mathcal{L}}
\newcommand{\N}{\mathbb{N}}
\renewcommand{\P}{\mathcal{P}}
\newcommand{\R}{\mathbb{R}}
\newcommand{\rad}{\operatorname{rad}}

\newcommand\ssb{\hskip-.25mm}
\newcommand\ssf{\hskip.25mm}

\newcommand{\T}{\mathbb{T}}
\newcommand{\W}{\mathcal{W}}
\newcommand{\Z}{\mathbb{Z}}

\begin{document}

\maketitle

\centerline{
To appear in
\textit{Trends in Harmonic Analysis\/}
}\centerline{
(\textit{dedicated to A. Fig\`a--Talamanca
on the occasion of his retirement\/})
}\centerline{
M.A. Picardello (ed.),
Springer INdAM Ser.,
Springer--Verlag, pp. 1--24
}

\begin{abstract}
We solve explicitly the shifted wave equation
\begin{equation*}
\partial_{\ssf t}^{\ssf2}u(x,t)=
(\Delta_{\ssf x}\ssb+\tfrac{Q^2}4)\ssf u(x,t)
\end{equation*}
on Damek--Ricci spaces,
using \'Asgeirsson's theorem
and the inverse dual Abel transform.
As an application, we investigate Huygens' principle.
A similar analysis is carried out
in the discrete setting of homogeneous trees.
\end{abstract}

\section{Introduction}
\label{Section1}

In the book \cite{H1} Helgason uses
\'Asgeirsson's mean value theorem (see Theorem II.5.28)
to solve the wave equation
\begin{equation}\label{WaveEquationRn}\begin{cases}
\,\partial_{\ssf t}^{\ssf2}u(x,t)
=\Delta_{\ssf x}\ssf u(x,t)\ssf,\\
\,u(x,0)=f(x)\ssf,\,
\partial_{\ssf t}|_{\ssf t=0}\,u(x,t)=g(x)\ssf,\\
\end{cases}\end{equation}
on Euclidean spaces \ssf$\R^d$
(see Exercise II.F.1 
and its solution pp. 574--575)
and the shifted wave equation
\begin{equation}\label{WaveEquationHnR}\begin{cases}
\,\partial_{\ssf t}^{\ssf2}u(x,t)
=\bigl(\Delta_{\ssf x}\ssb
+\ssb[\frac{d-1}2]^2\bigr)\ssf u(x,t)\ssf,\\
\,u(x,0)=f(x)\ssf,\,
\partial_{\ssf t}|_{\ssf t=0}\,u(x,t)=g(x)\ssf,\\
\end{cases}\end{equation}
on real hyperbolic spaces $H^d\ssf(\R)$
(see Exercise II.F.2 
and its solution pp. 575--577).
In this work we extend this approach
both to Damek--Ricci spaces and to homogeneous trees.
Along the way we clarify the role of the inverse dual Abel transform
in solving the shifted wave equation.

Recall that Damek--Ricci spaces are Riemannian manifolds,
which contain all hyperbolic spaces
$H^d\ssf(\R)$, $H^d\ssf(\C)$,
$H^d\ssf(\mathbb{H})$, $H^2(\mathbb{O})$
as a small subclass
and share nevertheless several features with these spaces.
Before \cite{H1}
the shifted wave equation \eqref{WaveEquationHnR}
on $H^d\ssf(\R)$
was solved explicitly in \cite[Section 7]{LP}.
Other hyperbolic spaces were dealt with
in \cite{BO,IV1,IV2}
and Damek--Ricci spaces in \cite{N}.
All these approaches are awkward in our opinion.
On one hand,
\cite{LP}, \cite{BO} and \cite{IV1,IV2}
rely on the method of descent i.e.~on shift operators,
which reduce the problem to checking formulae in low dimensions.
Moreover \cite{BO} involves classical compact dual symmetric spaces
and doesn't cover the exceptional case.
On the other hand,
\cite{N} involves complicated computations
and follows two different methods\,:
Helgason's approach for hyperbolic spaces
and heat kernel expressions \cite{ADY}
for general Damek--Ricci spaces.
In comparison we believe that
our presentation is simpler and more conceptual.

Several other works deal
with the shifted wave equation \eqref{WaveEquationHnR}
without using explicit solutions.
Let us mention \cite{BOS} (see also \cite[Section V.5]{H2})
for Huygens' principle and the energy equipartition
on Riemannian symmetric spaces of the noncompact type.
This work was extended to Damek--Ricci spaces in \cite{AD},
to Ch\'ebli--Trim\`eche hypergroups in \cite{EY}
and to  the trigonometric Dunkl setting in \cite{B, A}.
The nonlinear shifted wave equation
was studied in \cite{T, APV1, APV2},
first on real hyperbolic spaces and next on Damek--Ricci spaces.
These works involve sharp dispersive and Strichartz estimates
for the linear equation.
Related $L^p\!\to\!L^p$ estimates were obtained
in \cite{I} on hyperbolic spaces.

Our paper is organized as follows.
In Section 2, we review Damek--Ricci spaces
and spherical analysis thereon.
We give in particular explicit expressions for the Abel transform,
its dual and the inverse transforms.
In Section \ref{Section3}
we extend \'Asgeirsson's mean value theorem to Damek--Ricci spaces,
apply it to solutions to the shifted wave equation
and deduce explicit expressions, using the inverse dual Abel transform.
As an application, we investigate Huygens' principle.
Section \ref{Section4} deals
with the shifted wave equation on homogeneous trees,
which are discrete analogs of hyperbolic spaces.

Most of this work was done several years ago.
The results on Damek-Ricci spaces were cited in \cite{R}
and we take this opportunity to thank Fran{\c c}ois Rouvi\`ere
for mentioning them and for encouraging us to publish details.
We are also grateful to Nalini Anantharaman
for pointing out to us the connection
between our discrete wave equation \eqref{WaveEquationTree} on trees
and recent works \cite{BL1, BL2} of Brooks and Lindenstrauss.

\section{Spherical analysis on Damek--Ricci spaces}
\label{Section2}

We shall be content with a brief review about Damek--Ricci spaces
and we refer to the lecture notes \cite{R} for more information.

Damek--Ricci spaces are solvable Lie groups
\ssf$S\!=\!N\hspace{-.75mm}\rtimes\!A$\ssf,
which are extensions of Heisenberg type groups \ssf$N$
by $A\!\cong\!\R$
and which are equipped with a left-invariant Riemannian structure.
At the Lie algebra level,
\vspace{-3mm}
\begin{equation*}
\mathfrak{s}\;\equiv\;\underbrace{\R^m\,\oplus
\overbrace{\vphantom{\big|}\R^k
}^{\vphantom{\big|}\textstyle\mathfrak{z}}
}_{\textstyle\vphantom{|}\mathfrak{n}}
\oplus\!\underbrace{\vphantom{\big|}\R
}_{\textstyle\vphantom{|}\mathfrak{a}}\,,
\end{equation*}
\vspace{-5mm}

\noindent
with Lie bracket
\begin{equation*}
[(X,Y,z),(X'\!,Y'\!,z^{\ssf\prime})]
=(\tfrac z2\ssf X'\!
-\tfrac{z^{\ssf\prime}}2X,
\ssf z\,Y'\!-z^{\ssf\prime}\ssf Y+[\ssf X,X'],
0)
\end{equation*}
and inner product
\begin{equation*}
\langle(X,Y,z),(X'\!,Y'\!,z^{\ssf\prime})\rangle=
\langle\ssf X,X'\rangle_{\R^m}
+\langle\ssf Y,Y'\rangle_{\R^k}
+z\,z^{\ssf\prime}.
\end{equation*}
At the Lie group level,
\vspace{-2mm}
\begin{equation*}
S\;\equiv\;\underbrace{\R^m\,\times
\overbrace{\vphantom{\big|}\R^k
}^{\textstyle\vphantom{|}Z}}_{\textstyle\vphantom{|}N}
\times\underbrace{\vphantom{\big|}\R
}_{\textstyle\vphantom{|}A}\,,
\end{equation*}
\vspace{-5mm}

\noindent
with multiplication
\begin{equation*}
(X,Y,z)\cdot(X'\!,Y'\!,z^{\ssf\prime})
=(X\!+\ssb e^{\frac z2}X'\ssb,
Y\!+\ssb e^{\ssf z}\,Y'\!
+\ssb\tfrac12\,e^{\frac z2}\ssf[\ssf X,X'],
\ssf z\ssb+\ssb z^{\ssf\prime})\ssf.
\end{equation*}
So far \ssf$N$ could be
any simply connected nilpotent Lie group
of step $\le\ssb2$\ssf.
Heisenberg type groups are characterized by conditions
involving the Lie bracket and the inner product on $\mathfrak{n}$,
that we shall not need explicitly.
In particular \ssf$Z$ is the center of \ssf$N$
and \ssf$m$ \ssf is even.
One denotes by
\begin{equation*}
n=m\ssb+\ssb k\ssb+\!1
\end{equation*}
the (manifold) dimension of \ssf$S$
and by
\begin{equation*}
Q=\tfrac m2\!+\ssb k
\end{equation*}
the so--called homogeneous dimension of \ssf$N$.

Via the Iwasawa decomposition,
all hyperbolic spaces
$H^d\ssf(\R)$, $H^d\ssf(\C)$,
$H^d\ssf(\mathbb{H})$, $H^2(\mathbb{O})$
can be realized as Damek--Ricci spaces,
real hyperbolic spaces corresponding to
the degenerate case where \ssf$N$ is abelian.
But most Damek--Ricci spaces are not symmetric, although harmonic,
and thus provide numerous counterexamples
to the Lichnerowicz conjecture \cite{DR}.
Despite the lack of symmetry,
radial analysis on \ssf$S$
is similar to the hyperbolic space case
and fits into Jacobi function theory \cite{K}.

In polar coordinates,
the Riemannian volume on \ssf$S$ may be written as
\,$\delta(r)\ssf dr\ssf d\sigma$,
where
\vspace{-3mm}
\begin{align*}
\delta(r)
&=\,\overbrace{
2^{\ssf m+1}\ssf\pi^{\frac n2}\,\Gamma\bigl(\tfrac n2\bigr)^{-1}
}^{\vphantom{|}\const}\,
\bigl(\ssf\sinh\tfrac r2\bigr)^{\ssb m}\,
(\ssf\sinh r)^k\\
&=\,\underbrace{
2^{\ssf n}\ssf\pi^{\frac n2}\,\Gamma\bigl(\tfrac n2\bigr)^{-1}
}_{\vphantom{|}\const}\,
\bigl(\ssf\cosh\tfrac r2\bigr)^{\ssb k}\,
\bigl(\ssf\sinh\tfrac r2\bigr)^{\ssb n-1}
\end{align*}
is the common surface measure of all spheres of radius \ssf$r$ in $S$
and \ssf$d\sigma$ denotes the normalized surface measure
on the unit sphere in \ssf$\mathfrak{s}$\ssf.
We shall not need the full expression of
the Laplace-Beltrami operator $\Delta$ on $S$
but only its radial part
\begin{equation*}
\rad\Delta
=\bigl(\tfrac\partial{\partial r}\bigr)^{\ssb2}\!
+\underbrace{\bigl\{
\tfrac{n\ssb-\!1}2\ssf\coth\tfrac r2
+\tfrac k2\ssf\tanh\tfrac r2\ssf
\bigr\}}_{\textstyle\frac{\delta'(r)}{\delta(r)}}\ssf
\tfrac{\partial}{\partial r}
\end{equation*}
\vspace{-5mm}

\noindent
on radial functions and its horocyclic part
\begin{equation}\label{HorocyclicDR}
\Delta f=\bigl(\tfrac\partial{\partial z}\bigr)^{\ssb2}\ssb f
-\ssf Q\,\tfrac\partial{\partial z}\ssf f
\end{equation}
on \ssf$N$--\ssf invariant functions
i.e.~on functions \ssf$f\!=\!f(X,Y,z)$
depending only on $z$\ssf.
The Laplacian \ssf$\Delta$ \ssf commutes both
with left translations and with the averaging projector
\begin{equation*}
f^{\ssf\sharp}(r)
=\ssf\tfrac1{\delta(r)}\int_{S(e,r)}\!dx\,f(x)\,,
\end{equation*}
hence with all spherical means
\begin{equation*}
f_x^{\,\sharp}(r)
=\ssf\tfrac1{\delta(r)}\int_{S(x,r)}\!dy\,f(y)\,.
\end{equation*}
Thus
\begin{equation}\label{RadialDR}
(\Delta f)_x^{\ssf\sharp}
=(\rad\Delta\ssb)f_x^{\ssf\sharp}\,.
\end{equation}
Finally $\Delta$ has a spectral gap. More precisely
its \ssf$L^2$--\ssf spectrum is equal to the half--line
\ssf$\bigl(-\ssf\infty,-\frac{Q^2}4\ssf\bigr]$.

Radial Fourier analysis on \ssf$S$
may be summarized by the following commutative diagram
in the Schwartz space setting \cite{ADY}\,:
\vspace{1mm}
\begin{equation*}\label{diagramDR}\begin{aligned}
&\hspace{-2mm}\mathcal{S}(\R)_{\ssf\text{even}}\\
\mathcal{H}\!\nearrow\hspace{1.5mm}\approx\hspace{0mm}
&\hspace{6mm}\approx\hspace{1.5mm}\nwarrow\,\mathcal{F}\\
\mathcal{S}(S)^{\ssf\sharp}\hspace{7mm}
&\overset{\textstyle\approx}
{\underset{\textstyle\A\vphantom{\big|}}\longrightarrow}
\hspace{7mm}\mathcal{S}(\R)_{\ssf\text{even}}
\end{aligned}\end{equation*}
Here
\begin{equation*}
\mathcal{H}f(\lambda)
=\int_{\ssf S}\,dx\;\varphi_\lambda(x)\,f(x)
\end{equation*}
denotes the spherical Fourier transform on \ssf$S$,
\begin{equation*}
\A f(z)=\ssf e^{-\frac Q2\ssf z}\int_{\ssf\R^m}\!dX\int_{\ssf\R^k}\!dY\,f(X,Y,z)
\end{equation*}
the Abel transform,
\begin{equation*}
\F f(\lambda)=\int_{\ssf\R}dz\;e^{\,i\ssf\lambda\ssf z}\,f(z)
\end{equation*}
the classical Fourier transform on \ssf$\R$
\ssf and \,$\mathcal{S}(S)^\sharp$
\ssf the space of smooth radial functions
\ssf$f(x)\!=\!f(|x|)$ on \ssf$S$ \ssf such that
\begin{equation*}
\sup\nolimits_{\ssf r\ssf\ge\ssf0}\,
(1\!+\ssb r)^M\,e^{\ssf\frac Q2\ssf r}\;\bigl|
\bigl(\tfrac\partial{\partial r}\bigr)^{\ssb N}\hspace{-.6mm}f(r)
\hspace{.15mm}\bigr|<+\infty
\end{equation*}
for every \ssf$M\ssb,
\ssf N\hspace{-.75mm}\in\!\N$\hspace{.1mm}.
Recall that the Abel transform and its inverse
can be expressed explicitly in terms of Weyl fractional transforms,
which are defined by
\begin{align*}
\mathcal{W}_\mu^{\ssf\tau}f\ssf(r)
=\tfrac1{\Gamma(\mu\ssf+M)}
&\int_{\,r}^{+\infty}\!d\ssf(\cosh\tau s)\,
(\cosh\tau s\ssb-\ssb\cosh\tau r)^{\ssf\mu+M-1}\ssf
\bigl(-\tfrac d{d\ssf(\cosh\tau s)}\bigr)^{\ssb M}\ssb f\ssf(s)
\end{align*}
for \ssf$\tau\!>\!0$
\ssf and for \ssf$\mu\!\in\!\C$\ssf,
\ssf$M\hspace{-.75mm}\in\!\N$ \ssf such that
\ssf$\operatorname{Re}\mu\!>\!-M$.
Specifically,
\begin{equation*}
\A\ssf=\ssf c_1\,
\mathcal{W}_{m/2}^{\ssf1/2}\ssb\circ\mathcal{W}_{k/2}^{\,1}
\quad\text{and}\quad
\A^{-1}\ssb=\tfrac1{c_1}\,
\W_{-k/2}^{\,1}\ssb\circ\W_{-m/2}^{\ssf1/2}\,,
\end{equation*}
where \,$c_1\ssb=2^{\frac{3\ssf m+k}2}\,\pi^{\frac{m+k}2}$.
More precisely,
\begin{equation*}
\A^{-1}\ssb f\ssf(r)
=\tfrac1{c_1}\,\bigl(-\tfrac d{d\ssf(\cosh r)}\bigr)^{\!\frac k2}\ssf
\bigl(-\tfrac d{d\,(\cosh\frac r2)}\bigr)^{\!\frac m2}f\ssf(r)
\end{equation*}
if \ssf$n$ \ssf is odd i.e.~$k$ \ssf is even,
and
\begin{equation*}
\A^{-1}\ssb f\ssf(r)
=\tfrac1{c_1\ssf\sqrt{\pi\ssf}}
\int_{\,r}^{+\infty}\hspace{-1mm}
\tfrac{ds}{\sqrt{\ssf\cosh s\,-\,\cosh r\ssf}}\ssf
\bigl(-\tfrac d{ds}\bigr)\ssf
\bigl(-\tfrac d{d\ssf(\cosh s)}\bigr)^{\!\frac{k-1}2}\ssf
\bigl(-\tfrac d{d\ssf(\cosh\frac s2)}\bigr)^{\!\frac m2}f\ssf(s)
\end{equation*}
if \ssf$n$ \ssf is even i.e.~$k$ \ssf is odd.
Similarly, the dual Abel transform
\begin{equation}\label{DualAbelTransform}
\A^*\ssb f\ssf(r)
=\bigl(\ssf\widetilde{f}\,\bigr)^\sharp(r)\ssf,
\quad\text{where \,}
\widetilde{f}\ssf(X,Y,z)
=e^{\ssf\frac Q2\ssf z}\,f(z)\,,
\end{equation}
and its inverse can be expressed explicitly
in terms of Riemann-Liouville fractional transforms
\,$\mathcal{R}_\mu^{\ssf\tau}$\ssf,
which are defined by
\begin{equation*}
\mathcal{R}_\mu^{\ssf\tau}\ssf f\ssf(r)
=\tfrac1{\Gamma(\mu\ssf+M)}
\int_{\,0}^{\,r}\ssb d\ssf(\cosh\tau s)\,
(\cosh\tau r\ssb-\ssb\cosh\tau s)^{\ssf\mu+M-1}\,
\bigl(\ssf\tfrac d{d\ssf(\cosh\tau s)}\bigr)^{\ssb M}\ssb f\ssf(s)
\end{equation*}
for \ssf$\tau\!>\!0$
\ssf and for \ssf$\mu\!\in\!\C$\ssf,
\ssf$M\hspace{-.7mm}\in\!\N$ \ssf such that
\ssf$\operatorname{Re}\mu\!>\!-M$.

\begin{theorem}\label{TheoremDualAbelTransform}
The dual Abel transform \eqref{DualAbelTransform}
is a topological isomorphism between
\linebreak
\,$C^{\ssf\infty}\ssb(\R)_{\text{\rm even}}$
and
\,$C^{\ssf\infty}\ssb(S)^\sharp\ssb\equiv
C^{\ssf\infty}\ssb(\R)_{\text{\rm even}}$.
Explicitly,
\begin{equation*}
\A^*\ssb f\ssf(r)=\tfrac{c_{\ssf2}}2\,\bigl(\sinh\tfrac r2\bigr)^{\ssb-m}\ssf
(\sinh r)^{-(k-1)}\;\mathcal{R}_{\ssf k/2}^{\,1}\ssf
\bigl\{\bigl(\cosh\tfrac\cdot2\bigr)^{-1}\,\mathcal{R}_{m/2}^{\ssf1/2}\ssf
\bigl[\bigl(\sinh\tfrac\cdot2\bigr)^{-1}\ssb f\ssf\bigr]\bigr\}(r)
\end{equation*}
and
\begin{equation*}
(\A^*)^{-1}f(r)=\tfrac1{c_{\ssf2}}\,\tfrac d{dr}\,
\bigl(\mathcal{R}_{-m/2}^{\,1/2}\ssb\circ\mathcal{R}_{-k/2+1}^{\,1}\bigr)\ssf
\bigl\{\bigl(\sinh\tfrac\cdot2\bigr)^{\ssb m}\ssf
(\sinh\,\cdot\ssf)^{\ssf k-1}f\ssf\bigr\}\ssf(r)
\end{equation*}
where \,$c_{\ssf2}\ssb
=\frac{\vphantom{\big|}2^{\frac{n-1}2}\ssf\Gamma(\frac n2)}
{\vphantom{\big|}\sqrt{\pi\ssf}}
=\frac{\vphantom{\big|}(n\ssf-\ssf1)\ssf!}
{\vphantom{\big|}2^{\frac{n-1}2}\ssf\Gamma(\frac{n\ssf+1}2)}$\ssf.
More precisely,
\vspace{-2mm}
\begin{equation*}
(\A^*)^{-1}f(r)=\tfrac1{c_{\ssf2}}\,\tfrac d{dr}\,
\bigl(\ssf\tfrac d{d\ssf(\cosh\frac r2)}\bigr)^{\!\frac m2}\ssf
\bigl(\tfrac d{d\ssf(\cosh r)}\bigr)^{\!\frac k2-1}\ssf
\bigl\{\bigl(\sinh\tfrac r2\bigr)^{\ssb m}\ssf
(\sinh r)^{\ssf k-1}f(r)\bigr\}
\end{equation*}
if \ssf$n$ \ssf is odd i.e.~$k$ \ssf is even,
and
\begin{equation*}
(\A^*)^{-1}f(r)=\tfrac1{c_{\ssf2}\ssf\sqrt{\pi\ssf}}\,\tfrac d{dr}\,
\bigl(\tfrac d{d\ssf(\cosh\frac r2)}\bigr)^{\!\frac m2}\ssf
\bigl(\tfrac d{d\ssf(\cosh r)}\bigr)^{\!\frac{k-1}2}\!
\int_{\,0}^{\,r}\hspace{-1mm}
\tfrac{ds}{\sqrt{\ssf\cosh r\,-\,\cosh s\ssf}}\,
\bigl(\sinh\tfrac s2\bigr)^{\ssb m}\,
(\sinh s)^{\ssf k}\ssf f(s)
\end{equation*}
if \ssf$n$ \ssf is even i.e.~$k$ \ssf is odd.
\end{theorem}

\begin{proof}
Everything follows from the duality formulae
\begin{equation*}\begin{gathered}
\int_{\ssf\R}dr\,
\A f(r)\,g(r)\,
=\int_{\ssf S}\ssf dx\,
f(x)\,\A^*\ssb g(x)\,,\\
\int_{\,0}^{+\infty}\!d\ssf(\cosh\tau r)\,
\mathcal{W}_\mu^{\ssf\tau}\ssb f(r)\,g(r)\,
=\int_{\,0}^{+\infty}\!d\ssf(\cosh\tau r)\,
f(r)\;\mathcal{R}_\mu^{\ssf\tau}\ssf g(r)\,,
\end{gathered}\end{equation*}
and from the properties of the Riemann-Liouville transforms,
in particular
\begin{equation*}
\mathcal{R}_{1/2}^{\,\tau}:\,
r^{\ssf\ell}\,
C^{\ssf\infty}\ssb(\R)_{\text{even}}
\overset\approx\longrightarrow\,
r^{\ssf\ell+1}\,
C^{\ssf\infty}\ssb(\R)_{\text{even}}
\end{equation*}
for every integer \ssf$\ell\!\ge\!-1$\ssf.
\end{proof}

\begin{remark}
In the degenerate case \ssf$m\!=\!0$,
we recover the classical expressions
for real hyperbolic spaces \ssf$H^n(\R)$\ssf:
\begin{align*}
\A f(r)
&=\tfrac{(2\ssf\pi)^{\frac{n-1}2}}{\Gamma(\frac{n-1}2)}
\int_{\,r}^{+\infty}\!d\ssf(\cosh s)\,
(\cosh s\ssb-\ssb\cosh r)^{(n-3)/2}
f(s)\,,\\
\A^*\ssb f\ssf(r)
&=\ssf c_{\ssf3}\,(\sinh r)^{-(n-2)}\ssb
\int_{\,0}^{\,r}\!ds\,
(\cosh r\ssb-\ssb\cosh s)^{\frac{n-3}2}f(s)\,,
\end{align*}
where \,$c_{\ssf3}\ssb
=\frac{\vphantom{\big|}2^{\frac{n-1}2}\ssf\Gamma(\frac n2)}
{\vphantom{\big|}\sqrt{\pi\ssf}\,\Gamma(\frac{n-1}2)}
=\frac{\vphantom{\big|}(n\ssf-\ssf2)\ssf!}
{\vphantom{\big|}2^{\frac{n-3}2}\ssf\Gamma(\frac{n-1}2)^2}$\ssf,
\begin{align*}
\A^{-1}\ssb f\ssf(r)
&=(2\ssf\pi)^{-\frac{n-1}2}\,
\bigl(-\tfrac d{d\ssf(\cosh r)}\bigr)^{\!\frac{n-1}2}f(r)\,,\\
(\A^*)^{-1}\ssb f\ssf(r)
&=\tfrac{2^{\frac{n-1}2}\ssf(\frac{n-1}2)\ssf!}
{\vphantom{\frac00}(n\ssf-\ssf1)\ssf!}\,
\tfrac d{dr}\,
\bigl(\ssf\tfrac d{d\ssf(\cosh r)}\bigr)^{\!\frac{n-3}2}\ssf
\bigl\{(\sinh r)^{\ssf n-2}\ssf f(r)\bigr\}
\end{align*}
if \,$n$ \ssf is odd and
\begin{align*}
\A^{-1}\ssb f\ssf(r)
&=\tfrac{\vphantom{\frac00}1}{2^{\frac{n-1}2}\ssf\pi^{\frac n2}}
\int_{\,r}^{+\infty}\hspace{-1mm}
\tfrac{ds}{\sqrt{\ssf\cosh s\,-\,\cosh r\ssf}}\,
\bigl(-\tfrac d{ds}\bigr)\ssf
\bigl(-\tfrac d{d\ssf(\cosh s)}\bigr)^{\!\frac n2-1}f(s)\,,\\
(\A^*)^{-1}\ssb f(r)
&=\tfrac{\vphantom{\frac00}1}{2^{\frac{n-1}2}\ssf(\frac n2-1)\ssf!}\,
\tfrac d{dr}\,
\bigl(\ssf\tfrac d{d\ssf(\cosh r)}\bigr)^{\!\frac n2-1}\!
\int_{\,0}^{\,r}\ssb
\tfrac{ds}{\sqrt{\ssf\cosh r\,-\,\cosh s\ssf}}\,
(\sinh s)^{n-1}\ssf f(s)
\end{align*}
if \,$n$ \ssf is even.
\end{remark}

\section{\'Asgeirsson's mean value theorem
and the shifted wave equation\\
on Damek--Ricci spaces}
\label{Section3}

\begin{theorem}\label{AsgeirssonTheorem1DR}
Assume that \,$U\hspace{-1mm}\in\!
C^{\ssf\infty}\ssb(S\!\times\!S)$
\ssf satisfies
\begin{equation}\label{AsgeirssonHypothesisDR}
\Delta_{\ssf x}\,U(x,y)=\Delta_{\ssf y}\,U(x,y)\,.
\end{equation}
Then
\begin{equation}\label{AsgeirssonMean1DR}
\int_{S(x,r)}\!dx^{\ssf\prime}
\int_{S(y,s)}\!dy^{\ssf\prime}\;
U(x^{\ssf\prime}\!,y^{\ssf\prime})\,=\ssf
\int_{S(x,s)}\!dx^{\ssf\prime}
\int_{S(y,r)}\!dy^{\ssf\prime}\;
U(x^{\ssf\prime}\!,y^{\ssf\prime})
\end{equation}
for every \,$x,y\!\in\!S$ \ssf and \,$r,s\!>\!0$\ssf.
\end{theorem}

\noindent
The \textit{proof\/} is similar to
the real hyperbolic space case \cite[Section II.5.6]{H1}
once one has introduced the double spherical means
\begin{equation*}\label{DoubleSphericalMeansDR}
U_{\ssf x,y}^{\,\sharp,\sharp}\ssf(r,s)
=\ssf\tfrac1{\delta(r)}\int_{S(x,r)}\!dx^{\ssf\prime}\;
\tfrac1{\delta(s)}\int_{S(y,s)}\!dy^{\ssf\prime}\;
U(x^{\ssf\prime}\!,y^{\ssf\prime})
\end{equation*}
and transformed \eqref{AsgeirssonHypothesisDR} into
\begin{equation*}\label{AsgeirssonHypothesisRadialDR}
(\ssf\rad\Delta)_{\ssf r}\;
U_{\ssf x,y}^{\,\sharp,\sharp}\ssf(r,s)\,=\,
(\ssf\rad\Delta)_{\ssf s}\;
U_{\ssf x,y}^{\,\sharp,\sharp}\ssf(r,s)\,.
\end{equation*}
\vspace{-12mm}

\qed
\bigskip

\'Asgeirsson's Theorem is the following limit case
of Theorem \ref{AsgeirssonTheorem1DR},
which is obtained by dividing \eqref{AsgeirssonMean1DR} by $\delta(s)$
and by letting \,$s\ssb\to\ssb0$\ssf.

\begin{corollary}\label{AsgeirssonTheorem2DR}
Under the same assumptions,
\begin{equation*}\label{AsgeirssonMean2DR}
\int_{S(x,r)}\!dx^{\ssf\prime}\;
U(x^{\ssf\prime}\!,y)\,
=\,\int_{S(y,r)}\!dy^{\ssf\prime}\;
U(x,y^{\ssf\prime})\,.
\end{equation*}
\end{corollary}

Given a solution \ssf$u\!\in\!
C^{\ssf\infty}\ssb(S\!\times\!\R)$
to the shifted wave equation
\begin{equation}\label{Dalembertian}
\partial_{\ssf t}^{\ssf2}u(x,t)
=\bigl(\Delta_{\ssf x}\!+\ssb\tfrac{Q^2}4\bigr)\ssf u(x,t)
\end{equation}
on \ssf$S$ with initial data \ssf$u(x,0)\!=\!f(x)$
and \ssf$\partial_{\ssf t}|_{\ssf t=0}
\,u(x,t)\!=\ssb0$,
\ssf set
\begin{equation}\label{U}\textstyle
U(x,y)=e^{\ssf\frac Q2\ssf t}\,u(x,t),
\end{equation}
where \ssf$t$ \ssf is
the \ssf$z$ \ssf coordinate of \ssf$y$.
Then \eqref{U} satisfies \eqref{AsgeirssonHypothesisDR},
according to \eqref{HorocyclicDR}.
By applying Corollary \ref{AsgeirssonTheorem2DR} to \eqref{U}
with \ssf$y\!=\!e$ \ssf and \ssf$r\!=\!|t|$,
we deduce that the dual Abel transform of
\ssf$t\ssb\mapsto\ssb u(x,t)$,
as defined in \eqref{DualAbelTransform},
is equal to the spherical mean
\ssf$\displaystyle f_x^{\ssf\sharp}(|t|)$
of the initial datum~\ssf$f$.
Hence
\begin{equation*}
u(x,t)=(\A^*)^{-1}\ssb\bigl(f_x^{\ssf\sharp}\bigr)(t)\ssf.
\end{equation*}
By integrating with respect to time, we obtain the solutions
\begin{equation*}
u(x,t)\,=\int_{\,0}^{\,t}ds\;(\A^*)^{-1}\ssb\bigl(g_x^\sharp\bigr)(s)
\end{equation*}
to \eqref{Dalembertian}
with initial data \ssf$u(x,0)\!=\!0$
and \ssf$\partial_{\ssf t}|_{\ssf t=0}
\,u(x,t)\!=\ssb g(x)$.
In conclusion, general solutions to the shifted wave equation
\begin{equation}\label{WaveEquationDR}\begin{cases}
\,\partial_{\ssf t}^{\ssf2}u(x,t)
=\bigl(\Delta_{\ssf x}\!+\ssb\tfrac{Q^2}4\bigr)\ssf u(x,t)\\
\,u(x,0)=f(x)\ssf,\,
\partial_{\ssf t}|_{\ssf t=0}\,u(x,t)=g(x)\\
\end{cases}\end{equation}
on \ssf$S$ are given by
\begin{equation*}\label{GeneralSolutionDR}
u(x,t)\ssf=\ssf(\A^*)^{-1}\ssb\bigl(f_x^{\ssf\sharp}\bigr)(t)\ssf
+\int_{\,0}^{\,t}ds\;(\A^*)^{-1}\ssb\bigl(g_x^\sharp\bigr)(s)\ssf.
\end{equation*}
By using Theorem \ref{TheoremDualAbelTransform},
we deduce the following explicit expressions.

\begin{theorem}\label{SolutionWaveEquationDR}
{\rm (a)} When \ssf$n$ \ssf is odd,
the solution to \eqref{WaveEquationDR} is given by
\begin{equation*}\begin{aligned}\label{SolutionOddDimension}
u(x,t)&=\ssf c_{\ssf4}\,
\tfrac\partial{\partial t}\,
\bigl(\ssf\tfrac\partial{\partial\ssf(\cosh\frac t2)}\bigr)^{\!\frac m2}\ssf
\bigl(\ssf\tfrac\partial{\partial\ssf(\cosh t)}\bigr)^{\!\frac k2-1}\,
\Bigl\{\ssf\tfrac1{\sinh t}\int_{S(x,\ssf|t|)}\!dy\hspace{.75mm}f(y)\Bigr\}\\
&\,+\ssf c_{\ssf4}\,
\bigl(\ssf\tfrac\partial{\partial\ssf(\cosh\frac t2)}\bigr)^{\!\frac m2}\ssf
\bigl(\ssf\tfrac\partial{\partial\ssf(\cosh t)}\bigr)^{\!\frac k2-1}\,
\Bigl\{\ssf\tfrac1{\sinh t}\int_{S(x,\ssf|t|)}\!dy\;g(y)\Bigr\}
\ssf,
\end{aligned}\end{equation*}
with \,$c_{\ssf4}\ssb
=2^{-\frac{3\hspace{.1mm}m+k}2-1}\ssf\pi^{-\frac{n-1}2}$\ssf.

\noindent
{\rm (b)} When \ssf$n$ \ssf is even,
the solution to \eqref{WaveEquationDR} is given by
\begin{align*}\label{SolutionEvenDimension}
u(x,t)&=\ssf c_{\ssf5}\,
\tfrac\partial{\partial|t|}\,
\bigl(\ssf\tfrac\partial{\partial\ssf(\cosh\frac t2)}\bigr)^{\!\frac m2}\ssf
\bigl(\ssf\tfrac\partial{\partial\ssf(\cosh t)}\bigr)^{\!\frac{k-1}2}\ssb
\int_{B(x,\ssf|t|)}\!dy\;
\tfrac{f(y)}{\sqrt{\ssf\cosh t\,-\,\cosh d\ssf(y,\ssf x)\ssf}}\\
&\,+\ssf c_{\ssf5}\,\operatorname{sign}(t)\,
\bigl(\ssf\tfrac\partial{\partial\ssf(\cosh\frac t2)}\bigr)^{\!\frac m2}\ssf
\bigl(\ssf\tfrac\partial{\partial\ssf(\cosh t)}\bigr)^{\!\frac{k-1}2}\ssb
\int_{B(x,\ssf|t|)}\!dy\;
\tfrac{g(y)}{\sqrt{\ssf\cosh t\,-\,\cosh d\ssf(y,\ssf x)\ssf}}\,,
\end{align*}
with \,$c_{\ssf5}\ssb
=2^{-\frac{3\hspace{.1mm}m+k}2-1}\ssf\pi^{-\frac n2}$\ssf.
\end{theorem}

\begin{remark}
These formulae extend to the degenerate case \ssf$m\!=\!0$,
which corresponds to real hyperbolic spaces $H^n(\R)$\,{\rm:}

\noindent
{\rm (a)} \,$n$ \ssf odd\,{\rm:}
\begin{align*}
u(t,x)&=\ssf c_{\ssf6}\,\tfrac{\partial}{\partial t}\,
\bigl(\ssf\tfrac{\partial}{\partial\ssf(\cosh t)}\bigr)^{\!\frac{n-3}2}\,
\Bigl\{\ssf\tfrac1{\sinh t}\int_{S(x,\ssf|t|)}\!dy\hspace{.75mm}f(y)\Bigr\}\\
&\,+\ssf c_{\ssf6}\,
\bigl(\ssf\tfrac{\partial}{\partial\ssf(\cosh t)}\bigr)^{\!\frac{n-3}2}\ssf
\Bigl\{\ssf\tfrac1{\sinh t}\int_{S(x,\ssf|t|)}\!dy\;g(y)\Bigr\}\ssf,
\end{align*}
with \,$c_{\ssf6}\ssb=2^{-\frac{n+1}2}\ssf\pi^{-\frac{n-1}2}$\ssf.

\noindent
{\rm (b)} \,$n$ \ssf even\,{\rm:}
\begin{align*}
u(t,x)&=\ssf c_{\ssf7}\,\tfrac{\partial}{\partial|t|}\,
\bigl(\ssf\tfrac{\partial}{\partial\ssf(\cosh t)}\bigr)^{\!\frac n2-1}
\int_{B(x,\ssf|t|)}\!dy\;
\tfrac{f(y)}{\sqrt{\ssf\cosh t\,-\,\cosh d\ssf(y,\ssf x)\ssf}}\\
&\,+\ssf c_{\ssf7}\,\operatorname{sign}(t)\,
\bigl(\ssf\tfrac{\partial}{\partial\ssf(\cosh t)}\bigr)^{\!\frac n2-1}\!
\int_{B(x,\ssf|t|)}\!dy\;
\tfrac{g(y)}{\sqrt{\ssf\cosh t\,-\,\cosh d\ssf(y,\ssf x)\ssf}}\,,
\end{align*}
with \,$c_{\ssf7}\ssb=2^{-\frac{n+1}2}\ssf\pi^{-\frac n2}$\ssf.
\end{remark}

As an application,
let us investigate the propagation of solutions \ssf$u$
\ssf to the shifted wave equation \eqref{WaveEquationDR}
with initial data \ssf$f\ssb,\ssf g$
\ssf supported in a ball \,$B(x_{\ssf0},R\ssf)$.
The following two statements are immediate consequences
of Theorem \ref{SolutionWaveEquationDR}.
Firstly, waves propagate at unit speed.

\begin{corollary}\label{FinitePropagationSpeedDR}
Under the above assumptions,
\begin{equation*}
\operatorname{supp}\ssf u
\subset\{\ssf(x,t)\!\in\!S
\mid d\ssf(x,x_{\ssf0})\!\le\!|t|\!+\ssb R\,\}\,.
\end{equation*}
\end{corollary}

Secondly, Huygens' principle holds in odd dimension, as in the Euclidean setting.
This phenomenon was already observed in \cite{N}.

\begin{corollary}\label{HuygensPrincipleDR}
Assume that \ssf$n$ \ssf is odd.
Then, under the above assumptions,
\begin{equation*}
\operatorname{supp}\ssf u\subset\{\ssf(x,t)\!\in\!S\mid
|t|\!-\ssb R\ssb\le\ssb d\ssf(x,x_{\ssf0})\!\le\!|t|\!+\ssb R\,\}\,.
\end{equation*}
\end{corollary}

In even dimension, \,$u(x,t)$ may not vanish
when \ssf$d\ssf(x,x_{\ssf0})\!<\!|t|\!-\ssb R$\ssf,
but it tends asymptotically to \ssf$0$.
This phenomenon was observed in several settings,
for instance on Euclidean spaces in \cite{S},
on Riemannian symmetric spaces of the noncompact type \cite{BOS},
on Damek--Ricci spaces \cite{AD},
for Ch\'ebli--Trim\`eche hypergroups \cite{EY}, ...
Our next result differs from \cite{BOS, EY, AD} in two ways.
On one hand, we use explicit expressions instead of the Fourier transform.
On the other hand, we aim at energy estimates as in \cite{S},
which are arguably more appropriate than pointwise estimates.
Recall indeed that the total energy
\begin{equation}\label{TotalEnergyDR}
\E(t)=\ssf\K(t)+\P(t)
\end{equation}
is time independent, where
\begin{equation*}\label{KineticEnergyDR}
\K(t)=\tfrac12\int_{\ssf S}\,dx\;|\partial_{\ssf t}u(x,t)|^2
\end{equation*}
is the kinetic energy and
\begin{align*}\label{PotentialEnergyDR}
\P(t)
&=\tfrac12\int_{\ssf S}\,dx\;
\bigl(-\Delta_{\ssf x}\!-\!\tfrac{Q^2}4\bigr)
u(x,t)\;\overline{u(x,t)}\\
&=\tfrac12\int_{\ssf S}\,dx\;
\bigl\{\ssf|\ssf\nabla_{\ssb x}\ssf u(x,t)|^2\!
-\ssb\tfrac{Q^2}4\ssf|\ssf u(x,t)|^2\bigr\}
\end{align*}
the potential energy.
By the way, let us mention that
the equipartition of \eqref{TotalEnergyDR}
into kinetic and potential energies
was investigated in \cite{BOS}
and in the subsequent works \cite{EY, A, AD}
(see also \cite[Section V.5.5]{H2}
and the references cited therein).

\begin{lemma}\label{PointwiseEstimateDR}
Let \ssf$u$ \ssf be
a solution to \eqref{WaveEquationDR}
\ssf with smooth initial data
\ssf$f\!,\ssf g$
\ssf supported in a ball \,$B(x_{\ssf0},R\ssf)$.
Then
\begin{equation*}
u(x,t)\ssf,\,\partial_{\ssf t}u(x,t)\ssf,\,\nabla_{\ssb x}\ssf u(x,t)
\hspace{2mm}\text{are}\hspace{2mm}\text{\rm O}\bigl(
e^{\hspace{.2mm}-(Q/2)\ssf|t|}\ssf\bigr)
\end{equation*}
for every \ssf$x\!\in\!S$ \ssf and \,$t\!\in\!\R$ \ssf
such that \,$d\ssf(x,x_{\ssf0})\!\le\ssb|t|\!-\ssb R\!-\!1$\ssf.
\end{lemma}

\begin{proof}
Assume \,$t\!>\!0$ \ssf and consider the second part
\begin{equation}\begin{aligned}\label{Solution2WaveEquationDR}
v(x,t)&=
\bigl(\ssf\tfrac\partial{\partial\ssf(\cosh\frac t2)}\bigr)^{\!\frac m2}\ssf
\bigl(\ssf\tfrac\partial{\partial\ssf(\cosh t)}\bigr)^{\!\frac{k-1}2}
\int_{B(x,t)}\!dy\;
\tfrac{g(y)}{\sqrt{\ssf\cosh t\,-\,\cosh d\ssf(y,\ssf x)\ssf}}
\end{aligned}\end{equation}
of the solution \ssf$u(x,t)$
in Theorem \ref{SolutionWaveEquationDR}.b.
The case \,$t\!<\!0$ \ssf and the first part are handled similarly.
As \,$B(x_{\ssf0},R\ssf)\!\subset\ssb B(x,t)$,
we have
\begin{equation*}
\int_{B(x,t)}\!dy\;
\tfrac{g(y)}{\sqrt{\ssf\cosh t\,-\,\cosh d\ssf(y,\ssf x)\ssf}}\,
=\int_{B(x_{\ssf0},R\ssf)}\!dy\;
\tfrac{g(y)}{\sqrt{\ssf\cosh t\,-\,\cosh d\ssf(y,\ssf x)\ssf}}
\end{equation*}
and thus it remains to apply the differential operator
\begin{equation*}
D_{\ssf t}=
\bigl(\ssf\tfrac\partial{\partial\ssf(\cosh\frac t2)}\bigr)^{\!\frac m2}\ssf
\bigl(\ssf\tfrac\partial{\partial\ssf(\cosh t)}\bigr)^{\!\frac{k-1}2}
\end{equation*}
to \,$\bigl\{\cosh t\ssb-\cosh d\ssf(y,x)\bigr\}^{-\frac12}$\ssf.
Firstly
\begin{equation*}
\bigl(\ssf\tfrac\partial{\partial\ssf(\cosh t)}\bigr)^{\!\frac{k-1}2}\ssf
\bigl\{\cosh t\ssb-\cosh d\ssf(y,x)\bigr\}^{-\frac12}
=\,\operatorname{const.}\hspace{.5mm}
\bigl\{\cosh t\ssb-\cosh d\ssf(y,x)\bigr\}^{-\frac k2}
\end{equation*}
and secondly
\begin{align*}
&\bigl(\ssf\tfrac\partial{\partial\ssf(\cosh\frac t2)}\bigr)^{\!\frac m2}\ssf
\bigl\{\cosh t\ssb-\cosh d\ssf(y,x)\bigr\}^{-\frac k2}\\
&=\ssf\sum\nolimits_{\ssf0\ssf\le\ssf j\ssf\le\frac m4}
a_j\,\bigl(\cosh\tfrac t2\bigr)^{\!\frac m2-2j}\,
\bigl\{\cosh t\ssb-\cosh d\ssf(y,x)\bigr\}^{-\frac{m+k}2+j}\ssf,
\end{align*}
for some constants \ssf$a_j$.
As
\begin{equation*}
\cosh t\ssb-\cosh d\ssf(y,x)
=2\,\sinh\tfrac{t\ssf+\ssf d\ssf(y,\ssf x)}2\,
\sinh\tfrac{t\ssf-\ssf d\ssf(y,\ssf x)}2
\asymp e^{\;t},
\end{equation*}
we conclude that
\,$D_{\ssf t}\ssf\bigl\{\ssf\cosh t\ssb-\cosh d\ssf(y,x)\bigr\}^{-\frac12}$
\ssf and hence \ssf$v(x,t)$ are \,$\text{O}\bigl(e^{-\frac Q2\ssf t}\bigr)$.
The de\-ri\-va\-ti\-ves \ssf$\partial_{\ssf t}\ssf v(x,t)$
and \ssf$\nabla_{\ssb x}\ssf v(x,t)$
are estimated similarly.
As far as \ssf$\nabla_{\ssb x}v(x,t)$ is concerned,
we use in addition that
\begin{equation*}\textstyle
\sinh d\ssf(y,x)=\text{O}\ssf(e^{\;t})
\quad\text{and}\quad
|\ssf\nabla_{\ssb x}\ssf d\ssf(y,x)|
\ssb\le\ssb1.
\end{equation*}
This concludes the proof of Lemma \ref{PointwiseEstimateDR}.
\end{proof}

\begin{theorem}\label{AsymptoticHuygensDR}
Let \ssf$u$ \ssf be a solution
to \eqref{WaveEquationDR} with initial data
$f\ssb,g\!\in\ssb C_c^{\ssf\infty}(S)$
and let \ssf$R\ssb=\ssb R(t)$
be a positive function such that
\begin{equation*}\begin{cases}
\,R(t)\to+\infty\\
\,R(t)=\ssf\text{\rm o}\ssf(|t|)\\
\end{cases}
\quad\text{as \;}t\ssb\to\ssb\pm\infty\ssf.
\end{equation*}
Then
\begin{equation*}
\int_{\ssf d\ssf(x,\ssf e)\ssf
<\ssf|t|-\ssf R(t)}dx\;
\bigl\{\,|\ssf u(x,t)|^2+\ssf
|\ssf\nabla_{\ssb x}\ssf u(x,t)|^2
+\ssf|\partial_{\ssf t}u(x,t)|^2\ssf\bigr\}
\end{equation*}
tend to \ssf$0$ \ssf as \,$t\!\to\!\pm\infty$.
In other words, the energy of \,$u$
\ssf concentrates asymptotically inside the spherical shell
\begin{equation*}
\{\,x\!\in\!S\mid|t|\!-\!R(t)\ssb\le d\ssf(x,e)\ssb\le\ssb|t|\!+\!R(t)\ssf\}\,.
\end{equation*}
\end{theorem}

\noindent
\textit{Proof of Theorem \ref{AsymptoticHuygensDR}\/}.
\ssf By combining Lemma \ref{PointwiseEstimateDR}
with the volume estimate
\begin{equation*}
\operatorname{vol}\ssf B\bigl(e,|t|\!-\!R(t)\bigr)
\asymp\ssf e^{\,Q\ssf\{|t|-R(t)\}}
\quad\text{as \;}t\ssb\to\ssb\pm\infty\ssf,
\end{equation*}
we deduce that the three integrals
\begin{equation*}\begin{aligned}
&\int_{\ssf d\ssf(x,\ssf e)
\ssf<\ssf|t|-\ssf R(t)
}dx\;|\ssf u(x,t)|^2\,,\\
&\int_{\ssf d\ssf(x,\ssf e)
\ssf<\ssf|t|-\ssf R(t)
}\!dx\;|\ssf\nabla_{\ssb x}\ssf u(x,t)|^2\,,\\
&\int_{\ssf d\ssf(x,\ssf e)
\ssf<\ssf|t|-\ssf R(t)
}\!dx\;|\partial_{\ssf t}u(x,t)|^2
\end{aligned}\end{equation*}
are \ssf$\text{O}
\bigl(\hspace{.1mm}e^{\hspace{.2mm}-\hspace{.2mm}Q\,R(t)}\bigr)$
\ssf and hence tend to \,$0$
\,as \,$t\!\to\ssb\pm\infty$\ssf.
\qed

\section{The shifted wave equation on homogeneous trees}
\label{Section4}

This section is devoted to a discrete setting,
which is similar to the continuous setting considered so far.
A homogeneous tree \ssf$\T\!=\!\T_q$
of degree \ssf$q\ssb+\!1\!>\!2$
is a connected graph with no loops
and with the same number \ssf$q\ssb+\!1$
of edges at each vertex.
We shall be con\-tent with a brief review
and we refer to the expository paper \cite{CMS} for more information
(see also the monographs \cite{FP, FN}).

For the counting measure,
the volume of any sphere \ssf$S(x,n)$
in \ssf$\T$ \ssf is given by
\begin{equation*}
\delta(n)\ssf=\,\begin{cases}
\;1
&\text{if \,}n\ssb=\ssb0\ssf,\\
\,(\ssf q\ssb+\!1)\,q^{\ssf n-1}
&\text{if \,}n\!\in\!\N^*.
\end{cases}\end{equation*}
Once we have chosen an origin \ssf$0\!\in\!\T$ and
a geodesic \ssf$\omega\ssb:\ssb\Z\ssb\to\!\T$ \ssf through \ssf$0$,
let us denote by $|x|\hspace{-.7mm}\in\!\N$
\ssf the distance of a vertex \ssf$x\!\in\!\T$ to the origin
and by \ssf$h(x)\hspace{-.6mm}\in\!\Z$
\ssf its horocyclic height (see Figure \ref{UHSPT}).

\begin{figure}
\psfrag{0}[c]{$0$}
\psfrag{1}[c]{$1$}
\psfrag{2}[c]{$2$}
\psfrag{-1}[c]{$-1$}
\psfrag{h}[c]{$h$}
\psfrag{omega}[c]{$\omega$}
\includegraphics[height=70mm]{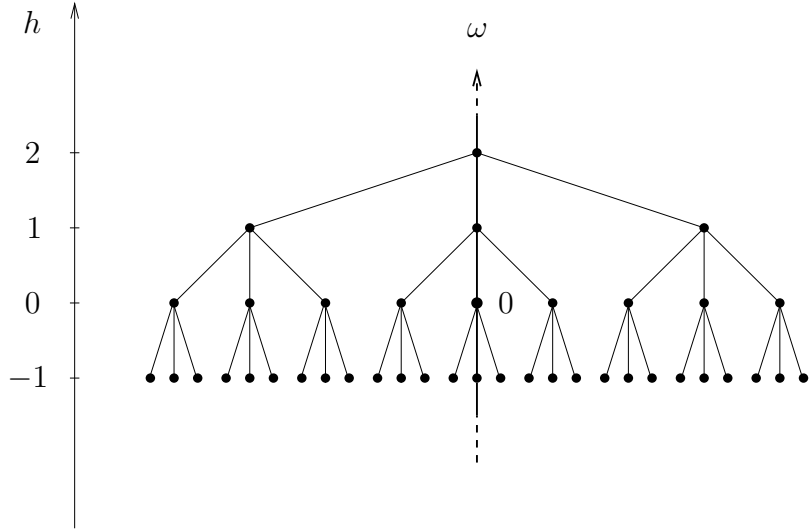}
\caption{Upper half--space picture of \,$\T_3$}
\label{UHSPT}
\end{figure}

The combinatorial Laplacian is defined
on \ssf$\Z$ \ssf by
\begin{equation*}\label{LaplacianZ}
\L^{\ssf\Z}f\ssf(n)
=f(n)-\ssf\tfrac{f(n\ssf+\ssf1)\,+\,f(n\ssf-\ssf1)}2\,,
\end{equation*}
and similarly on \ssf$\T$ \ssf by
\begin{equation}\label{LaplacianTree}
\L^{\ssf\T}\ssb f\ssf(x)
=f(x)-\ssf\tfrac1{q\,+\ssf1}\ssf
\sum\nolimits_{\ssf y\in S(x,1)}\ssb
f(y)\ssf.
\end{equation}
The \ssf$L^2$--\ssf spectrum
of \ssf$\L^{\ssf\T}$
is equal to the interval
\ssf$[\ssf1\!-\ssb\gamma,
1\!+\ssb\gamma\ssf]$,
\ssf where
\begin{equation*}
\gamma\ssf
=\tfrac2{q^{\ssf1/2}\ssf+\,q^{-1/2}}
\in(0,1)\ssf.
\end{equation*}
We have
\begin{equation}\label{RadialTree}
\L^{\ssf\T}\ssb f\ssf(n)\,
=\,\begin{cases}
\;f(0)-f(1)
&\text{if \,}n\ssb=\ssb0\ssf\\
\,f(n)-\frac1{q\,+\ssf1}\,f(n\ssb-\!1)-\frac q{q\,+\ssf1}\,f(n\ssb+\!1)
&\text{if \,}n\!\in\!\N^*
\end{cases}\end{equation}
on radial functions and
\begin{align}
\L^{\ssf\T}
\ssb f\ssf(h)\ssf
&=\ssf f(h)-\tfrac q{q\,+\ssf1}\,f(h\ssb-\!1)
-\tfrac1{q\,+\ssf1}\,f(h\ssb+\!1)
\nonumber\\
&\label{HorocyclicTree}
=\,\gamma\;q^{\frac h2}\,
\L_{\,h}^{\ssf\Z}\ssf
\bigl\{q^{-\frac h2}\ssf f(h)\bigr\}
+(1\!-\ssb\gamma)\,f(h)
\end{align}
on horocyclic functions.

Again, radial Fourier analysis on \ssf$\T$
may be summarized by the following commutative diagram
\begin{equation*}\label{diagramTree}\begin{aligned}
\hspace{-2mm}C^{\ssf\infty}\ssb(
&\R\ssf/\tau\ssf\Z
)_{\ssf\text{even}}\\
\mathcal{H}\!\nearrow\hspace{1.5mm}\approx\hspace{0mm}
&\hspace{6mm}\approx\hspace{1.5mm}\nwarrow\,\mathcal{F}\\
\mathcal{S}(\T)^{\ssf\sharp}\hspace{7mm}
&\overset{\textstyle\approx}
{\underset{\textstyle\A\vphantom{\big|}}\longrightarrow}
\hspace{7mm}\mathcal{S}(\Z)_{\ssf\text{even}}
\end{aligned}\end{equation*}
Here
\begin{equation*}
\H f(\lambda)
=\sum\nolimits_{\ssf x\in\T}
\varphi_\lambda(x)\,f(x)
\qquad\forall\;\lambda\!\in\!\R
\end{equation*}
denotes the spherical Fourier transform on \ssf$\T$,
\begin{equation*}
\A f(h)
=\ssf q^{\frac h2}\ssf
\sum\nolimits_{\hspace{-1.5mm}\substack{
\vphantom{o}\\x\in\T\\h(x)=h}}\hspace{-1mm}
f(|x|)
\qquad\forall\;h\!\in\!\Z
\end{equation*}
the Abel transform and
\begin{equation*}
\F f(\lambda)=\sum\nolimits_{\ssf h\in\Z}
q^{\,i\ssf\lambda\ssf h}\,f(h)
\qquad\forall\;\lambda\!\in\!\R
\end{equation*}
a variant of the classical Fourier transform on \ssf$\Z$\ssf.
Moreover \,$\tau\ssb=\ssb\frac{2\ssf\pi}{\log q}$,
\,$\mathcal{S}(\Z)_{\text{even}}$ \ssf denotes
the space of even functions on \ssf$\Z$
\ssf such that
\begin{equation*}\textstyle
\sup_{\ssf n\in\N^*}\ssf
n^{\ssf k}\,|f(n)|<+\infty
\qquad\forall\;k\!\in\!\N\ssf,
\end{equation*}
and \,$\mathcal{S}(\T)^\sharp$
\ssf the space of radial functions on \ssf$\T$
\ssf such that
\begin{equation*}\textstyle
\sup_{\ssf n\in\N^*}\ssf
n^{\ssf k}\,q^{\frac n2}\,|f(n)|<+\infty
\qquad\forall\;k\!\in\!\N\ssf.
\end{equation*}
Consider finally the dual Abel transform
\begin{equation*}
\A^*\ssb f(n)\ssf
=\ssf\tfrac1{\delta(n)}\,
\sum\nolimits_{\!\substack{
\vphantom{o}\\x\in\T\\|x|=n}}\ssb
q^{\frac{h(x)}2}\ssf f\bigl(h(x)\bigr)
\qquad\forall\;n\!\in\!\N\ssf.
\end{equation*}
The following expressions are obtained by elementary computations.

\begin{lemma}\label{AbelTree}
{\rm (a)}
The Abel transform is given by
\begin{equation*}\begin{aligned}
\A f(h)\ssf
&=\,q^{\frac{|h|}2}\ssf f(|h|)
+\tfrac{q\,-\ssf1}q\ssf
\sum\nolimits_{\ssf k=1}^{+\infty}
q^{\frac{|h|}2+\ssf k}\ssf
f(|h|\!+\ssb2\ssf k)\\
&=\sum\nolimits_{\ssf k=0}^{+\infty}
q^{\frac{|h|}2+\ssf k}\ssf
\bigl\{\ssf f(|h|\!+\ssb2\ssf k)
-f(|h|\!+\ssb2\ssf k
\ssb+\ssb2)\bigr\}
\qquad\forall\;h\!\in\!\Z
\end{aligned}\end{equation*}
and the dual Abel transform by
\begin{equation*}
\A^*\!f(n)
=\tfrac{2\,q}{q\,+\ssf1}\,
q^{-\frac{|n|}2}f(\pm\ssf n)
+\tfrac{q\,-\ssf1}{q\,+\ssf1}\,q^{-\frac{|n|}2}\ssf
\sum\nolimits_{\hspace{-6mm}\substack{
\vphantom{o}\\-|n|\,<\,k\,<\,|n|\\
\text{\rm$k$ \ssf has same parity as $n$}}}
\hspace{-6mm}f(\pm\ssf k)
\end{equation*}
\vspace{-4mm}

\noindent
if \,$n\!\in\!\Z^*$\ssb, resp. $\A^*\!f(0)\ssb=\!f(0)$\ssf.

\noindent
{\rm (b)}
The inverse Abel transform is given by
\begin{equation*}\begin{aligned}
\A^{-1}\ssb f(n)\ssf
&=\ssf\sum\nolimits_{\ssf k=0}^{+\infty}\ssf q^{-\frac n2-\ssf k}\ssf
\bigl\{\ssf f(n\ssb+\ssb2\ssf k)-f(n\ssb+\ssb2\ssf k\ssb+\ssb2)\bigr\}\\
&=\,q^{-\frac n2}f(n)-(q\ssb-\!1)\ssf\sum\nolimits_{\ssf k=1}^{+\infty}\ssf
q^{-\frac n2-\ssf k}\ssf f(n\ssb+\ssb2\ssf k)
\qquad\forall\;n\!\in\!\N
\end{aligned}\end{equation*}
and the inverse dual Abel transform by
\begin{equation*}\begin{aligned}
(\A^*)^{-1}\ssb f(h)\ssf
&=\ssf\tfrac12\,q^{\frac h2}f(h)
+\tfrac12\,q^{-\frac h2}f(1)\\
&\,+\ssf\tfrac12\ssf
\sum\nolimits_{\ssf k=1}^{\frac{h-1}2}
q^{\frac h2-2\hspace{.1mm}k+1}\ssf
\bigl\{\ssf f(h\ssb-\ssb2\ssf k\ssb+\ssb2)-f(h\ssb-\ssb2\ssf k)\bigr\}\\
&=\tfrac{q^{\ssf1/2}\ssf+\,q^{-1/2}}2\,q^{\frac{h-1}2}f(h)
-\tfrac{q\,-\,q^{-1}}2\,q^{-\frac h2}\ssf
\sum\nolimits_{\ssf0\ssf<\,k\;\text{\rm odd}\,<\,h}
\ssf q^{\,k}\ssf f(k)
\end{aligned}\end{equation*}
if \,$h\!\in\!\N$ \ssf is odd,
respectively
\begin{equation*}\textstyle
(\A^*)^{-1}\ssb f\ssf(0)=f\ssf(0)
\end{equation*}
and
\begin{equation*}\begin{aligned}
(\A^*)^{-1}\ssb f(h)\ssf
&=\ssf\tfrac12\,q^{\frac h2}f(h)+\tfrac12\,q^{-\frac h2}f(0)\\
&\,+\ssf\tfrac12\ssf
\sum\nolimits_{\ssf k=1}^{\,\frac h2}
q^{\frac h2-2\hspace{.1mm}k+1}\ssf
\bigl\{\ssf f(h\ssb-\ssb2\ssf k\ssb+\ssb2)-f(h\ssb-\ssb2\ssf k)\bigr\}\\
&=\ssf\tfrac{q^{1/2}\ssf+\,q^{-1/2}}2\;q^{\frac{h-1}2}f(h)
-\tfrac{q^{1/2}\ssf-\,q^{-1/2}}2\;q^{-\frac{h-1}2}f(0)\\
&\,-\ssf\tfrac{q\,-\,q^{-1}}2\,q^{-\frac h2}\ssf
\sum\nolimits_{\ssf0\ssf<\,k\;\text{\rm even}\,<\,h}
\ssf q^{\,k}\ssf f(k)
\end{aligned}\end{equation*}
if \,$h\!\in\!\N^*$ is even.
\end{lemma}

We are interested in
the following shifted wave equation on \ssf$\T$\,:
\begin{equation}\label{WaveEquationTree}\begin{cases}
\,\gamma\,\L_{\,n}^{\ssf\Z}\ssf u(x,n)
=\bigl(\L_{\,x}^{\ssf\T}\!
-\!1\!+\ssb\gamma\bigr)\ssf u(x,n)\ssf,\\
\,u(x,0)\ssb=\ssb f(x)\ssf,\;
\{\ssf u(x,1)-u(x,-1)\}/\ssf2\ssb=\ssb g(x)\ssf.\\
\end{cases}\end{equation}
As was pointed out to us by Nalini Anantharaman,
this equation occurs in the recent works \cite{BL1, BL2}.
The unshifted wave equation with discrete time was studied in \cite{CP}
and the shifted wave equation with continuous time in \cite{MS}.

We will solve \eqref{WaveEquationTree}
by applying the following discrete version of \'Asgeirsson's mean value theorem
and by using the explicit expression of the inverse dual Abel transform.

\begin{theorem}\label{AsgeirssonTheoremTree}
Let \,$U$ \ssf be a function on \ssf$\T$
\ssf such that
\begin{equation}\label{AsgeirssonHypothesisTree}
\L_{\,x}^{\ssf\T}\,U(x,y)\ssf
=\,\L_{\,y}^{\ssf\T}\,U(x,y)
\qquad\forall\;x,y\!\in\!\T\ssf.
\end{equation}
Then
\begin{equation*}\label{AsgeirssonMean1Tree}
\sum\nolimits_{\ssf x^{\ssf\prime}\ssb\in S(x,m)}
\sum\nolimits_{\ssf y^{\ssf\prime}\ssb\in S(y,n)}\!
U(x^{\ssf\prime}\!,y^{\ssf\prime})\,=\,
\sum\nolimits_{\ssf x^{\ssf\prime}\ssb\in S(x,n)}
\sum\nolimits_{\ssf y^{\ssf\prime}\ssb\in S(y,m)}\!
U(x^{\ssf\prime}\!,y^{\ssf\prime})
\end{equation*}
for every \ssf$x,y\!\in\!\T$
and \ssf$m,n\!\in\!\N$.
\ssf In particular
\begin{equation}\label{AsgeirssonMean2Tree}
\sum\nolimits_{\ssf x^{\ssf\prime}\ssb\in S(x,n)}\!
U(x^{\ssf\prime}\!,y)\,=\,
\sum\nolimits_{\ssf y^{\ssf\prime}\ssb\in S(y,n)}\!
U(x,y^{\ssf\prime})\,.
\end{equation}
\end{theorem}

In order to prove Theorem \ref{AsgeirssonTheoremTree},
we need the following discrete analog of \eqref{RadialDR}.

\begin{lemma}\label{LemmaAsgeirssonTree}
Consider the spherical means
\begin{equation*}\label{SphericalMeanTree}
f_x^{\,\sharp}(n)=\tfrac1{\delta(n)}\ssf
\sum\nolimits_{\ssf y\in S(x,n)}\ssb f(y)
\qquad\forall\;x\!\in\!\T,\;
\forall\;n\!\in\!\N\ssf.
\end{equation*}
Then
\begin{equation*}
(\ssf\L^{\ssf\T}\ssb f\ssf)_{\ssf x}^{\,\sharp}\ssf(n)\,
=\,(\ssf\rad\L\ssf)_n^{\vphantom{\sharp}}\;
f_{\ssf x}^{\,\sharp}(n)\ssf,
\end{equation*}
where \,$\rad\L$
\ssf denotes the radial part \eqref{RadialTree}
of \,$\L^{\ssf\T}$.
\end{lemma}

\noindent
\textit{Proof of Lemma \ref{LemmaAsgeirssonTree}\/}.
We have
\begin{equation*}
(\ssf\L^{\ssf\T}\ssb f\ssf)_{\ssf x}^{\,\sharp}\ssf(n)\,
=\,\begin{cases}
\;f(x)-f_{\ssf x}^{\,\sharp}(1)
&\text{if \,}n\ssb=\ssb0\ssf,\\
\;f_{\ssf x}^{\,\sharp}(n)
-\frac1{q\,+\ssf1}\ssf f_{\ssf x}^{\,\sharp}(n\ssb-\!1)
-\frac q{q\,+\ssf1}\ssf f_{\ssf x}^{\,\sharp}(n\ssb+\!1)
&\text{if \,}n\!\in\!\N^*\ssb.\\
\end{cases}
\end{equation*}
\vspace{-9.5mm}

\qed
\bigskip

\noindent
\textit{Proof of Theorem \ref{AsgeirssonTheoremTree}\/}.
\ssf Fix \ssf$x,y\!\in\!\T$
and consider the double spherical means
\begin{equation*}\label{DoubleSphericalMeansTree}
U_{\ssf x,y}^{\,\sharp,\sharp}\ssf(m,n)
=\ssf\tfrac1{\delta(m)}\ssf
\sum\nolimits_{\ssf\substack{
\vphantom{o}\\x^{\ssf\prime}\ssb\in S(x,m)}}
\tfrac1{\delta(n)}\ssf
\sum\nolimits_{\ssf\substack{
\vphantom{o}\\y^{\ssf\prime}\ssb\in S(y,n)}}
U(x^{\ssf\prime}\!,y^{\ssf\prime})\,,
\end{equation*}
that we shall denote by \ssf$V\ssb(m,n)$ for simplicity.
According to Lemma \ref{LemmaAsgeirssonTree},
our assumption \eqref{AsgeirssonHypothesisTree} may be rewritten as
\begin{equation}\label{AsgeirssonHypothesisRadialTree}
(\ssf\rad\L\ssf)_m\,
V\ssb(m,n)
=(\ssf\rad\L\ssf)_n\,
V\ssb(m,n)\,.
\end{equation}
Let us prove the symmetry
\begin{equation}\label{SymmetryV}
V\ssb(m,n)=\ssf V\ssb(n,m)
\qquad\forall\;m,n\!\in\!\N
\end{equation}
by induction on \,$\ell\ssb=\ssb m\ssb+\ssb n$\ssf.
First of all,
\eqref{SymmetryV} is trivial if \ssf$\ell\ssb=\ssb0$
\ssf and \eqref{SymmetryV} with \ssf$\ell\ssb=\!1$
is equivalent to \eqref{AsgeirssonHypothesisRadialTree}
with \ssf$m\ssb=\ssb n\ssb=\ssb0$\ssf.
Assume next that \ssf$\ell\ssb\ge\!1$
and that \eqref{SymmetryV} holds for
\ssf$m\ssb+\ssb n\ssb\le\ell\ssf$.
On one hand,
let \ssf$m\!>\!n\!>\!0$ \ssf
with \,$m\ssb+\ssb n\ssb=\ell\ssb+\!1$
and let $1\!\le\ssb k\ssb\le\ssb m\ssb-\ssb n$\ssf.
We deduce from \eqref{AsgeirssonHypothesisRadialTree}
at the point $(m\ssb-\ssb k\ssf,n\ssb+\ssb k\ssb-\!1)$
that
\begin{equation}\begin{aligned}\label{Equation1}
&\;V\ssb(m\ssb-\ssb k\ssb+\!1,\ssf n\ssb+\ssb k\ssb-\!1)
-\ssf V\ssb(m\ssb-\ssb k\ssf,\ssf n\ssb+\ssb k)\ssf=\\
&=\ssf q\;\{\,V\ssb(m\ssb-\ssb k\ssf,\ssf n\ssb+\ssb k\ssb-\ssb2)
-\ssf V\ssb(m\ssb-\ssb k\ssb-\!1,\ssf n\ssb+\ssb k\ssb-\!1)\ssf\}\,.
\end{aligned}\end{equation}
By adding up \eqref{Equation1} over \ssf$k$\ssf,
we obtain
\begin{equation}\label{Equation2}
V\ssb(m,n)-\ssf V\ssb(n,m)
=\ssf q\,\{\,V\ssb(m\ssb-\!1,\ssf n\ssb-\!1)
-\ssf V\ssb(n\ssb-\!1,\ssf m\ssb-\!1)\ssf\}\,,
\end{equation}
which vanishes by induction.
On the other hand,
we deduce from \eqref{AsgeirssonHypothesisRadialTree}
at the points $(\ell,0)$ and $(0,\ell\ssf)$ that
\vspace{-1mm}
\begin{equation*}\begin{cases}
\;V\ssb(\ell\ssb+\!1,0)
=(q\ssb+\!1)\;V\ssb(\ell,1)-\ssf q\;V\ssb(\ell,0)\ssf,\\
\;V\ssb(0,\ell\ssb+\!1)
=(q\ssb+\!1)\;V\ssb(1,\ell\ssf)-\ssf q\;V\ssb(0,\ell\ssf)\ssf.
\end{cases}\end{equation*}
Hence \,$V\ssb(\ell\ssb+\!1,0)\ssb=\ssf V\ssb(0,\ell\ssb+\!1)$
\ssf by using \eqref{Equation2} and by induction.
This concludes the proof of Theorem \ref{AsgeirssonTheoremTree}.
\qed
\bigskip

Let us now solve explicitly
the shifted wave equation \eqref{WaveEquationTree} on $\T$
as we did in Section \ref{Section3} for
the shifted wave equation \eqref{WaveEquationDR} on Damek--Ricci spaces.
Consider first a solution \ssf$u$
\ssf to \eqref{WaveEquationTree}
with initial data \ssf$u(x,0)\!=\!f(x)$
\ssf and
\ssf$\{u(x,1)\!-\ssb u(x,-1)\}/\ssf2\ssb=0$\ssf.
On one hand,
as $(x,n)\ssb\mapsto\ssb u(x,-n)$
satisfies the same Cauchy problem, we have
\ssf$u(x,-n)\ssb=\ssb u(x,n)$
by uniqueness.
On the other hand,
according to \eqref{HorocyclicTree},
the function
\vspace{-1mm}
\begin{equation*}
U(x,y)=\ssf q^{\frac{h(y)}2}\,u(x,h(y))
\qquad\forall\;x,y\!\in\!\T
\end{equation*}
satisfies \eqref{AsgeirssonHypothesisTree}.
Thus, by applying \eqref{AsgeirssonMean2Tree} to \ssf$U$
with \ssf$y\ssb=\ssb0$\ssf,
we deduce that
the dual Abel trans\-form of \,$n\mapsto\ssb u(x,n)$
\ssf is equal to the spherical mean
\ssf$f_x^{\ssf\sharp}(n)$
of the initial datum $f$.
Hence
\begin{equation*}
u(x,n)=(\A^*)^{-1}\ssb\bigl(f_x^{\ssf\sharp}\bigr)(n)
\qquad\forall\;x\!\in\!\T,\;\forall\;n\!\in\!\N\ssf.
\end{equation*}
Consider next a solution \ssf$u$ \ssf to \eqref{WaveEquationTree}
with initial data \ssf$u(x,0)\!=\!0$ \ssf
and \ssf$\{u(x,1)\ssb-u(x,-1)\}/\ssf2$ $=g(x)$\ssf.
Then \ssf$u(x,n)$ \ssf is an odd function of \ssf$n$ \ssf and
\begin{equation*}
v(x,n)=\tfrac{u(x,\ssf n\,+\ssf1)\,-\,u(x,\ssf n\,-\ssf1)}2
\end{equation*}
is a solution to \eqref{WaveEquationTree}
with initial data \ssf$v(x,0)\!=\!g(x)$ \ssf
and \ssf$\{v(x,1)\ssb-v(x,-1)\}/\ssf2\ssb=\ssb0$\ssf.
Hence
\begin{equation*}
u(x,n)=\begin{cases}
\,2\,{\displaystyle
\sum\nolimits_{\,0\ssf<\ssf k\hspace{1mm}\text{odd}\,<\ssf n}
}v(x,k)
&\text{if \,$n\!\in\!\N^*$ is even}\ssf,\vphantom{\Big|}\\
\,g(x)\ssb+2\,{\displaystyle
\sum\nolimits_{\,0\ssf<\ssf k\hspace{1mm}\text{even}\,<\ssf n}}
v(x,k)
&\text{if \,$n\!\in\!\N^*$ is odd}\ssf,\vphantom{\Big|}\\
\end{cases}\end{equation*}
with \ssf$v(x,n)\ssb=\ssb(\A^*)^{-1}\ssb\bigl(g_x^{\ssf\sharp}\bigr)(n)$\ssf.
By using Lemma \ref{AbelTree}.b,
we deduce the following explicit expressions.

\begin{theorem}\label{SolutionWaveEquationTree}
The solution to \eqref{WaveEquationTree} is given by
\begin{equation*}\begin{aligned}
u(x,n)\ssf
&=\ssf\tfrac12\,q^{-\frac{|n|}2}
\sum\nolimits_{\,d\ssf(y,\ssf x)\ssf=\ssf|n|}
\ssb f(y)\,
-\,\tfrac{q\,-\ssf1}2\,q^{-\frac{|n|}2}
\sum\nolimits_{\substack{\vphantom{o}\\
\hspace{-1.5mm}d\ssf(y,\ssf x)\ssf<\,|n|\\
\hspace{-3mm}|n|-d\ssf(y,\ssf x)\hspace{1mm}\text{\rm even}}}
\hspace{-1mm}f(y)\\
&\,+\,\operatorname{sign}(n)\,q^{-\frac{|n|-1}2}
\sum\nolimits_{\hspace{-2.5mm}\substack{\vphantom{o}\\
d\ssf(y,\ssf x)\ssf<\,|n|\\
|n|-d\ssf(y,\ssf x)\hspace{1mm}\text{\rm odd}}}
\hspace{-2mm}g(y)
\qquad\forall\,x\!\in\!\T,\,\forall\,n\!\in\!\Z^*\!,
\end{aligned}\end{equation*}
In other words,
\vspace{-6mm}
\begin{equation}\label{Solution2WaveEquationTree}
u(x,n)=\ssf\overbrace{
\tfrac{M_{\ssf|n|}\ssf-\,M_{\ssf|n|-2}}{2\vphantom{|}}
}^{\textstyle\vphantom{\big|}C_{\hspace{.1mm}n}}
\ssf f\ssf(x)\ssf
+\ssf\overbrace{\vphantom{\tfrac||}
\operatorname{sign}(n)\,M_{\ssf|n|-1}
}^{\textstyle\vphantom{\big|}S_{\hspace{.1mm}n}}
\ssf g\ssf(x)\,,
\end{equation}
where
\begin{equation}\label{Mn}
M_{\ssf n}\ssf f\ssf(x)
=\ssf q^{-\frac n2}\ssf
\sum\nolimits_{\hspace{-2.75mm}\substack{\vphantom{o}\\
d\ssf(y,\ssf x)\ssf\le\,n\\
n\ssf-\ssf d\ssf(y,\ssf x)
\hspace{1mm}\text{\rm even}}}
\hspace{-2.5mm}f(y)
\end{equation}
\vspace{-5mm}

\noindent
if \,$n\!\ge\!0$
\ssf and \,$M_{-1}\!=\ssb0$\ssf.
\end{theorem}

\begin{remark}\label{CnSn}
Notice that the radial convolution operators
\ssf$C_{\hspace{.1mm}n}$ and
\ssf$S_{\hspace{.1mm}n}$ above
correspond, via the Fourier transform, to the multipliers
\begin{equation*}
\cos_{\ssf q}n\ssf\lambda
\quad\text{and}\quad
\tfrac{\sin_{\ssf q}\ssb n\ssf\lambda}{\sin_{\ssf q}\ssb\lambda}\,,
\end{equation*}
where \,$\cos_{\ssf q}\ssb\lambda\ssb
=\ssb\frac{q^{\ssf i\ssf\lambda}\ssf+\,q^{-i\ssf\lambda}}2$
\,and \,$\sin_{\ssf q}\ssb\lambda\ssb
=\ssb\frac{q^{\ssf i\ssf\lambda}\ssf-\,q^{-i\ssf\lambda}}{2\,i}$\ssf.
\end{remark}

As we did in Section \ref{Section3},
let us next deduce propagation properties
of solutions \ssf$u$
\ssf to the shifted wave equation \eqref{WaveEquationTree}
with initial data \ssf$f\ssb,\ssf g$
\ssf supported in a ball \,$B(x_{\ssf0},N)$.

\begin{corollary}\label{WavePropagationTree}
Under the above assumptions,
\begin{itemize}
\item[\rm (a)]
$\hspace{1.5mm}u(x,n)=\text{\rm O}\bigl(q^{-\frac{|n|}2}\bigr)$
\,$\forall\;x\!\in\!\T$,
$\forall\;n\!\in\!\Z$\ssf,
\item[\rm (b)]
$\hspace{1.5mm}\operatorname{supp}\ssf u\subset
\{\ssf(x,n)\!\in\ssb\T\!\times\!\Z
\mid d\ssf(x,x_{\ssf0})\!\le\!|n|\!+\ssb N\,\}\,.$
\end{itemize}
\end{corollary}

Obviously Huygens' principle doesn't hold for \eqref{WaveEquationTree},
strictly speaking.
Let us show that it holds asymptotically,
as for even dimensional Damek--Ricci spaces.
For this purpose, define as follows the kinetic energy
\begin{equation*}\label{KineticEnergyTree1}
\K(n)=\tfrac12\ssf\sum\nolimits_{\ssf x\in\T}\,
\bigl|\tfrac{u(x,n\ssf+1)\,-\,u(x,n\ssf-1)}2\bigr|^2
\end{equation*}
and the potential energy
\begin{equation}\begin{aligned}\label{PotentialEnergyTree1}
\P(n)\ssf
&=\,\tfrac1{4\,q}\,
\sum\nolimits_{\hspace{-1.5mm}\substack{
\vphantom{o}\\x,y\in\T\\d\ssf(x,y)=2}}
\bigl|\tfrac{u(x,\ssf n)\,-\,u(y,\ssf n)}2\bigr|^2
-\,\tfrac{(q\,-\ssf1)^2}{8\,q}\,
\sum\nolimits_{\ssf x\in\T}|\ssf
u(x,n)|^2\\
&=\,\tfrac{q\,+\ssf1}8\ssf
\sum\nolimits_{\ssf x\in\T}\ssf
\bigl(\widetilde{\L}_x\!-\ssb\widetilde\gamma\,\bigr)u(x,n)\;
\overline{u(x,n)}
\end{aligned}\end{equation}
for solutions \ssf$u$
\ssf to \eqref{WaveEquationTree}.
Here
\begin{equation*}
\widetilde{\L}\ssb f(x)
=f(x)
-\tfrac1{q\,(q\,+\ssf1)}\,
\sum\nolimits_{\ssf y\in S(x,2)}\ssf
f(y)
\end{equation*}
is the $2$--step Laplacian on \ssf$\T$
\ssf and
\begin{equation*}
\widetilde\gamma\,=\ssf
\tfrac{(q\,-\ssf1)^2}{q\,(q\,+\ssf1)}
\in(0,1)\ssf.
\end{equation*}

\begin{lemma}\label{EnergyTree}
{\rm(a)} The \ssf$L^2$--\ssf spectrum
of \,$\widetilde{\L}$ \ssf is equal to the interval
\,$\bigl[\,\widetilde\gamma,\frac{q\,+\ssf1}q\ssf\bigr]$.
Thus the potential energy \eqref{PotentialEnergyTree1} is nonnegative.

\noindent
{\rm(b)} The total energy
\begin{equation*}\label{TotalEnergyTree1}
\E(n)=\ssf\K(n)+\P(n)
\end{equation*}
is independent of \,$n\!\in\!\Z$\ssf.
\end{lemma}

\begin{proof}
(a) follows for instance from the relation
\begin{equation*}
\widetilde{\L}=\tfrac{q\,+\ssf1}q\,
\L^{\ssf\T}\ssf
(\ssf2-\ssb\L^{\ssf\T}\ssf)
\end{equation*}
and from the fact that
the \ssf$L^2$--\ssf spectrum
of \,$\L^{\ssf\T}$
\ssf is equal to the interval
\,$[\ssf1\!-\ssb\gamma,1\!+\ssb\gamma\ssf]$\ssf.

\noindent
(b) Notice that the shifted wave equation
\begin{equation*}
\gamma\,\L_{\,n}^{\ssf\Z}\ssf u(x,n)
=\bigl(\L_{\,x}^{\ssf\T}\!
-\!1\!+\ssb\gamma\bigr)\,u(x,n)
\end{equation*}
amounts to
\begin{equation*}
u(x,n\ssb+\!1)+u(x,n\ssb-\!1)
=\tfrac1{\sqrt{\ssf q\,}}\ssf
\sum\nolimits_{\ssf y\in S(x,1)}\ssb u(y,n)\,.
\end{equation*}
As
\begin{equation*}
\sum\nolimits_{\ssf x\in\T}
\sum\nolimits_{\ssf y,\ssf z\in S(x,1)}
u(y,n)\,\overline{u(z,n)}
=(q\ssb+\!1)\ssf
\sum\nolimits_{\ssf x\in\T}
|\ssf u(x,n)|^2
+\sum\nolimits_{\hspace{-1.5mm}\substack{
\vphantom{o}\\y,\ssf z\in\T\\
d\ssf(y,\ssf z)=2}}\hspace{-1.5mm}
u(y,n)\,\overline{u(z,n)}\,,
\end{equation*}
we have on one hand
\begin{equation}\label{KineticEnergyTree2}\begin{aligned}
\K(n)
&=\ssf\tfrac{q\,+\ssf1}{8\,q}\ssf
\sum\nolimits_{\ssf x\in\T}
|\ssf u(x,n)|^2
+\ssf\tfrac12\ssf
\sum\nolimits_{\ssf x\in\T}
|\ssf u(x,n\ssb\pm\!1)|^2\\
&\,+\ssf\tfrac1{8\,q}\ssf
\sum\nolimits_{\hspace{-1.5mm}\substack{
\vphantom{o}\\x,\ssf y\in\T\\d\ssf(x,\ssf y)=2}}\hspace{-2mm}
u(x,n)\,\overline{u(y,n)}
-\tfrac1{2\ssf\sqrt{\ssf q\,}}\ssf
\sum\nolimits_{\hspace{-1.5mm}\substack{
\vphantom{o}\\x,\ssf y\in\T\\
d\ssf(x,\ssf y)=1}}\hspace{-1.5mm}
\operatorname{Re}\ssf
\bigl\{u(x,n)\,\overline{u(y,n\ssb\pm\!1)}\bigr\}\,.
\end{aligned}\end{equation}
On the other hand,
\begin{equation}\label{PotentialEnergyTree2}
\P(n)=\ssf\tfrac{3\ssf q\,-\ssf1}{8\,q}\ssf
\sum\nolimits_{\ssf x\in\T}
|\ssf u(x,n)|^2
-\ssf\tfrac1{8\,q}\ssf
\sum\nolimits_{\hspace{-1.5mm}\substack{
\vphantom{o}\\x,\hspace{.2mm}y\in\T\\
d\ssf(x,\hspace{.2mm}y)=2}}\hspace{-1.5mm}
u(x,n)\,\overline{u(y,n)}\,.
\end{equation}
By adding up \eqref{KineticEnergyTree2}
and \eqref{PotentialEnergyTree2}, we obtain
\begin{equation*}\label{TotalEnergyTree2}\begin{aligned}
\E(n)&=\ssf\tfrac12\ssf
\sum\nolimits_{\ssf x\in\T}
|\ssf u(x,n)|^2
+\ssf\tfrac12\ssf
\sum\nolimits_{\ssf x\in\T}
|\ssf u(x,n\ssb\pm\!1)|^2\\
&\,-\ssf\tfrac1{2\,\sqrt{\ssf q\,}}\ssf
{\displaystyle\sum\nolimits_{\hspace{-1.5mm}\substack{
\vphantom{o}\\x,\ssf y\in\T\\
d\ssf(x,\ssf y)=1}}}\hspace{-1.5mm}
\operatorname{Re}\ssf
\bigl\{u(x,n)\,\overline{u(y,n\ssb\pm\!1)}\bigr\}
\end{aligned}\end{equation*}
and we deduce from this expression that
\begin{equation*}
\E(n)=\E(n\ssb\pm\!1)\ssf.
\end{equation*}
This concludes the proof of Lemma \ref{EnergyTree}.
\end{proof}

\begin{remark}\label{EnergyConservationTree}
Alternatively, Lemma \ref{EnergyTree}.b can be proved
by expressing the energies \ssf$\K(n)$, $\P(n)$, $\E(n)$
in terms of the initial data \ssf$f$\!, $g$
\ssf and by using spectral calculus.
Specifically,
\begin{equation*}\begin{aligned}
\K(n)
&=\ssf\tfrac18\ssf\sum\nolimits_{\ssf x\in\T}\ssf
(\ssf C_{n+1}\!-C_{n-1})^{\hspace{.1mm}2}f(x)\;\overline{f(x)}\\
&\,+\ssf\tfrac18\ssf\sum\nolimits_{\ssf x\in\T}\ssf
(\ssf S_{n+1}\!-S_{n-1})^{\hspace{.1mm}2}\ssf g(x)\;\overline{g(x)}\\
&\,+\ssf\tfrac14\ssf\operatorname{Re}\ssf
\sum\nolimits_{\ssf x\in\T}\ssf
(\ssf C_{n+1}\!-C_{n-1})\ssf
(\ssf S_{n+1}\!-S_{n-1})\ssf f(x)\;\overline{g(x)}
\end{aligned}\end{equation*}
and
\begin{equation*}\begin{aligned}
\P(n)
&=\ssf\tfrac14\ssf\sum\nolimits_{\ssf x\in\T}\ssf
(1\!-C_{\ssf2})\,C_n^{\,2}f(x)\;\overline{f(x)}\\
&\,+\ssf\tfrac14\ssf\sum\nolimits_{\ssf x\in\T}\ssf
(1\!-C_{\ssf2})\,S_n^{\,2}\ssf g(x)\;\overline{g(x)}\\
&\,+\ssf\tfrac12\ssf\operatorname{Re}\ssf
\sum\nolimits_{\ssf x\in\T}\ssf
(1\!-C_{\ssf2})\,C_n\,S_n\ssf f(x)\;\overline{g(x)}\,.
\end{aligned}\end{equation*}
Here we have used the fact that
\begin{equation*}
\tfrac{q\,+\ssf1}8\,
(\ssf\widetilde{\L}\ssb-\widetilde{\gamma}\hspace{.5mm})
=\tfrac18\,(\ssf3-\ssb M_{\ssf2})
=\tfrac14\,(1\ssb-\ssf C_{\ssf2})\,.
\end{equation*}
Hence
\begin{equation*}
\E(n)
=\sum\nolimits_{\ssf x\in\T}
U_n^+\ssb f(x)\,\overline{f(x)}\,
+\sum\nolimits_{\ssf x\in\T}
V_n^+\ssb g(x)\,\overline{g(x)}\,
+\ssf2\,\operatorname{Re}\ssf
\sum\nolimits_{\ssf x\in\T}
W_n^+\ssb f(x)\,\overline{g(x)}\,,
\end{equation*}
where
\begin{align*}
U_n^+\ssb&=\ssf\tfrac18\,
(\ssf C_{n+1}\!-C_{n-1})^{\hspace{.1mm}2}
+\ssf\tfrac14\,(1\!-C_{\ssf2})\,C_n^{\,2}\ssf,\\
V_n^+\ssb&=\ssf\tfrac18\,
(\ssf S_{n+1}\!-S_{n-1})^{\hspace{.1mm}2}
+\ssf\tfrac14\,(1\!-C_{\ssf2})\,S_n^{\,2}\ssf,\\
W_n^+\ssb&=\ssf\tfrac18\,
(\ssf C_{n+1}\!-C_{n-1})\ssf
(\ssf S_{n+1}\!-S_{n-1})
+\ssf\tfrac14\,(1\!-C_{\ssf2})\,C_n\,S_n\ssf.\\
\end{align*}
By considering the corresponding multipliers, we obtain
\begin{equation*}
U_n^+\ssb=\ssf\tfrac14\,(1\!-C_{\ssf2})\,,\quad
V_n^+\ssb=\ssf\tfrac12\,,\quad
W_n^+\ssb=\ssf0\,,
\end{equation*}
and we conclude that
\begin{equation*}
\E(n)
=\ssf\tfrac14\ssf
\sum\nolimits_{\ssf x\in\T}\ssf
(1\!-C_{\ssf2})\ssf f(x)\,\overline{f(x)}
+\ssf\tfrac12\ssf
\sum\nolimits_{\ssf x\in\T}
|g(x)|^2
=\ssf\E(0)\,.
\end{equation*}
\end{remark}

Let us turn to
the asymptotic equipartition of the total energy
\ssf$\E\!=\!\E(n)$.

\begin{theorem}\label{AsymptoticEquipartitionTree}
Let \ssf$u$ \ssf be
a solution to \eqref{WaveEquationTree}
with finitely supported initial data \ssf$f$ and \ssf$g$\ssf.
Then the kinetic energy \ssf$\K(n)$
and the potential energy \ssf$\P(n)$
tend both to \,$\E/\ssf2$
\,as \,$n\ssb\to\ssb\pm\infty$\ssf.
\end{theorem}

\begin{proof}
Let us show that the difference
\,$\K(n)\ssb-\!\P(n)$
\ssf tends to \ssf$0$\ssf.
By resuming the computations in Remark \ref{EnergyConservationTree},
we obtain
\begin{equation*}
\K(n)\ssb-\ssb\P(n)
=\sum\nolimits_{\ssf x\in\T}\ssb
U_n^-\ssb f(x)\,\overline{f(x)}\ssf
+\sum\nolimits_{\ssf x\in\T}\ssb
V_n^-\ssb g(x)\,\overline{g(x)}\ssf
+\ssf2\ssf\operatorname{Re}
\sum\nolimits_{\ssf x\in\T}\ssb
W_n^-\ssb f(x)\,\overline{g(x)}\,,
\end{equation*}
with
\begin{align*}
U_n^-\ssb&=\ssf\tfrac18\,
(\ssf C_{n+1}\!-C_{n-1})^{\hspace{.1mm}2}\ssb
-\tfrac14\,(1\!-C_{\ssf2})\,C_n^{\,2}
=-\,\tfrac14\,
(1\!-C_{\ssf2})\,C_{\ssf2\ssf n}\,,\\
V_n^-\ssb&=\ssf\tfrac18\,
(\ssf S_{n+1}\!-S_{n-1})^{\hspace{.1mm}2}\ssb
-\tfrac14\,(1\!-C_{\ssf2})\,S_n^{\,2}
=\ssf\tfrac12\,C_{\ssf2\ssf n}\,,\\
W_n^-\!&=\ssf\tfrac18\,
(\ssf C_{n+1}\!-C_{n-1})\,
(\ssf S_{n+1}\!-S_{n-1})
-\tfrac14\,(1\!-\ssb C_{\ssf2})\,C_n\,S_n
=-\ssf\tfrac14\,
(1\!-\ssb C_{\ssf2})\,S_{\ssf2\ssf n}\,.
\end{align*}
As
\begin{equation*}
\|\,C_{\ssf2\ssf n\ssf}
f\ssf\|_{\ssf\ell^{\ssf\infty}}\ssb
\le\tfrac{q\,-\ssf1}2\;q^{-|n|}\,
\|\ssf f\ssf\|_{\ssf\ell^{\ssf1}}
\quad\text{and}\quad
\|\ssf(1\!-\ssb C_{\ssf2})\ssf
f\ssf\|_{\ssf\ell^{\ssf1}}\ssb
\le\bigl\{\tfrac{q\,-\,q^{-1}\!}2\ssb+\ssb2\ssf\bigr\}\,
\|\ssf f\ssf\|_{\ssf\ell^{\ssf1}}\,,
\end{equation*}
the expression
\begin{equation*}
\sum\nolimits_{\ssf x\in\T}
U_n^-\ssb f(x)\,\overline{f(x)}\ssf
=-\,\tfrac14\ssf
\sum\nolimits_{\ssf x\in\T}
C_{\ssf2\ssf n\ssf}f(x)\,
\overline{(1\!-\ssb C_{\ssf2})f(x)}
\end{equation*}
tends to \ssf$0$.
The expressions
\begin{equation*}
\sum\nolimits_{\ssf x\in\T}\ssb
V_n^-\ssb g(x)\,\overline{g(x)}\ssf
=\,\tfrac12\ssf
\sum\nolimits_{\ssf x\in\T}
C_{\ssf2\ssf n\ssf}g(x)\,
\overline{g(x)}
\end{equation*}
and
\begin{equation*}
\sum\nolimits_{\ssf x\in\T}\ssb
W_n^-\ssb f(x)\,\overline{g(x)}\ssf
=-\,\tfrac14\ssf
\sum\nolimits_{\ssf x\in\T}
S_{\ssf2\ssf n\ssf}f(x)\,
\overline{(1\!-\ssf C_{\ssf2})f(x)}
\end{equation*}
are handled in the same way.
This concludes the proof of Theorem \ref{AsymptoticEquipartitionTree}.
\end{proof}



Let us conclude with the asymptotic Huygens principle.

\begin{theorem}\label{AsymptoticHuygensTree}
Let \ssf$u$ \ssf be
a solution to \eqref{WaveEquationTree}
with finitely supported initial data
and let \ssf$(N_n)_{n\in\Z}$
\ssf be a sequence of positive integers such that
\begin{equation*}\begin{cases}
\,N_n\ssb\to+\infty\\
\,N_n=\ssf\text{\rm o}\ssf(|n|)\\
\end{cases}
\quad\text{as \;}n\ssb\to\ssb\pm\infty\ssf.
\end{equation*}
Then the expressions
\begin{equation*}
\sum\nolimits_{\substack{
\vphantom{o}\\\hspace{-4mm}x\in\T\\
\hspace{-6mm}|x|\ssf<\ssf|n|-\ssf N_n}}\hspace{-4mm}
|\,u(x,n)\ssf|^2\,,\hspace{2mm}
\sum\nolimits_{\hspace{-6mm}\substack{
\vphantom{,}\\x,y\in\T\\
|x|,\ssf|y|\ssf<\ssf|n|-\ssf N_n\\
d\ssf(x,\ssf y)=2}}\hspace{-6mm}
|\,u(x,n)\!-\ssb u(y,n)\ssf|^2\,,\hspace{2mm}
\sum\nolimits_{\substack{
\vphantom{o}\\\hspace{-4mm}x\in\T\\
\hspace{-6mm}|x|\ssf<\ssf|n|-\ssf N_n}}\hspace{-4mm}
|\,u(x,n\ssb+\!1)\!-\ssb u(x,n\ssb-\!1)\ssf|^2
\end{equation*}
tend to \ssf$0$ \ssf as \,$n\!\to\!\pm\infty$.
In other words, the energy of \,$u$
\ssf concentrates asymptotically inside the spherical shell
\begin{equation*}
\{\,x\!\in\!\T\mid
|n|\!-\ssb N_n\ssb\le\ssb|x|\ssb\le\ssb|n|\!+\ssb N_n\ssf\}\,.
\end{equation*}
\end{theorem}

\noindent
The \textit{proof\/} is similar to the proof of
Theorem \ref{AsymptoticHuygensDR}.
\qed

\end{document}